\documentclass[11pt]{amsart}
\usepackage[linkcolor=blue, citecolor = blue]{hyperref}
\hypersetup{colorlinks=true}

\usepackage{amsthm,amsfonts,amsmath,amssymb,amscd} 

\usepackage{verbatim}
\usepackage{float}
\usepackage{wrapfig}
\usepackage{mathtools}
\usepackage[color=orange!20, linecolor=orange]{todonotes}

\usepackage{enumitem}

\usepackage{ytableau}

\usepackage{tikz}
\usetikzlibrary{decorations.markings}
\tikzstyle{vertex}=[circle, draw, inner sep=0pt, minimum size=4pt]

\usetikzlibrary{arrows,automata}
\usepackage{tikz, tikz-3dplot, pgfplots}
\usetikzlibrary[positioning,patterns]

\newtheorem{theorem}{Theorem}[section]
\newtheorem{proposition}[theorem]{Proposition}
\newtheorem{lemma}[theorem]{Lemma}
\newtheorem{corollary}[theorem]{Corollary}

\theoremstyle{definition}
\newtheorem{definition}[theorem]{Definition}
\newtheorem{example}[theorem]{Example}
\newtheorem{problem}[theorem]{Problem}

\theoremstyle{remark}
\newtheorem{remark}[theorem]{Remark}

\newcommand{\sgn}{\mathrm{sgn}}

\newcommand{\srank}{\mathrm{srank}}

\title[Tensor slice rank and Cayley's first hyperdeterminant]
{Tensor slice rank and Cayley's first hyperdeterminant} 

\author[Alimzhan Amanov \and Damir Yeliussizov]{Alimzhan Amanov \and Damir Yeliussizov}

\address{KBTU, Almaty, Kazakhstan}
\email{\href{mailto:alimzhan.amanov@gmail.com}{alimzhan.amanov@gmail.com}, \href{mailto:yeldamir@gmail.com}{yeldamir@gmail.com}}

\begin{document}

\begin{abstract}
Cayley's first hyperdeterminant is a straightforward generalization of determinants for tensors. We prove that nonzero hyperdeterminants imply lower bounds on some types of tensor ranks. This result applies to the slice rank introduced by Tao and more generally to partition ranks introduced by Naslund.  
As an application, we show upper bounds on 
some generalizations of colored sum-free sets based on constraints related to order polytopes.
\end{abstract}

\maketitle


\section{Introduction}
 
\subsection{Tensors and ranks} Let 
$V$ be a 
vector space over a field $\mathbb{F}$. {\it Tensors} are elements of the space $V^{\otimes d} = {V \otimes \cdots \otimes V}$ ($d$ times). For a basis $(\mathbf{e}_i)$ of $V$ each tensor $\mathsf{T} \in V^{\otimes d}$ has a coordinate representation  
$$
\mathsf{T} = \sum_{i_1,\ldots, i_d} T(i_1,\ldots, i_d)\, \mathbf{e}_{i_1} \otimes \ldots \otimes \mathbf{e}_{i_d}.  
$$
Let $\dim(V) = n$ and we identify $\mathsf{T}$ with 
the corresponding function $T : [n]^{d} \to \mathbb{F}$
which we call a 
{\it $d$-tensor}, where $[n] := \{1,\ldots, n \}$. The next terms will be defined for $d$-tensors 
(but one can also do that in a coordinate-free way). 
Let $\mathsf{T}^d(n) := \{T : [n]^{d} \to \mathbb{F} \}$ be the  set of $d$-tensors. 

\vspace{0.5em}

We are going to 
consider several types of {\it tensor ranks} which can be described via the following general definition. 
For a subset $\mathsf{S} \subset \mathsf{T}^d(n)$ 
we define the {\it rank function} $\mathrm{rank}_{\mathsf{S}} : \mathsf{T}^d(n) \to \mathbb{N}$ 
as 
$$
\mathrm{rank}_{\mathsf{S}}(T) := \min\left\{r\, :\, 
T = \sum_{\ell = 1}^{r} T_{\ell}, \quad T_{\ell} \in \mathsf{S}  
\right\}
$$
and $\mathrm{rank}_{\mathsf{S}}(T) = 0$ iff $T = 0$,
assuming each tensor in $\mathsf{T}^d(n)$ can be represented as a sum of tensors from ${\mathsf{S}}$. 
In particular, $\mathrm{rank}_{\mathsf{S}}(T) = 1$ for $T \in {\mathsf{S}}$. 
We say that tensors in $\mathsf{S}$ are {\it simple}. 

\

Consider examples of rank functions whose definitions are based on choice of simple tensors.

\subsubsection*{\bf The rank} 
For the {\it tensor rank}, 
simple tensors are fully decomposable, i.e.
$T \in \mathsf{S}$ 
iff it has a form 
$$
T(i_1,\ldots, i_d) = \mathbf{v}_1(i_1) \cdots \mathbf{v}_d(i_d), \quad \forall\, i_1,\ldots, i_d \in [n]
$$
for some nonzero vectors $\mathbf{v}_i \in \mathbb{F}^n$. The tensor rank  of $T \in \mathsf{T}^d(n)$ is denoted by $\mathrm{rank}(T)$. It can be easily seen that $\mathrm{rank}(T) \le n^{d-1}$ and (by counting argument) 
that for most tensors $\mathrm{rank}(T) \ge n^{d-1}/d$. At the same time, 
the ranks of explicit tensors are mostly unknown.  
The notion of tensor rank goes back to Hitchcock \cite{hitch}, and it has many applications, see e.g. \cite{lands} for more on the subject. 

\subsubsection*{\bf The slice rank} 
For the {\it slice rank}, 
simple tensors are 
decomposable along some coordinate, i.e.  $T \in \mathsf{S}$ iff it has a form
$$
T(i_1,\ldots, i_d)  = \mathbf{v}(i_{k}) \cdot T_1(i_1,\ldots, i_{k - 1}, i_{k + 1}, \ldots, i_{d}), \quad \forall\, i_1,\ldots, i_d \in [n]
$$
for some $k \in [d]$, nonzero vector $\mathbf{v} \in \mathbb{F}^n$ and a $(d-1)$-tensor $T_1 \in \mathsf{T}^{d-1}(n)$. 
(Note that each simple tensor can be decomposed along a different coordinate $k$.)
The slice rank of $T \in \mathsf{T}^d(n)$ is denoted by $\srank(T)$. We clearly have the inequality
$$
\srank(T) \le \mathrm{rank}(T).
$$
In addition, since each $T \in \mathsf{T}^d(n)$ can always be expressed as the following sum of simple tensors 
$$
T(i_1,\ldots, i_d) = \sum_{\ell = 1}^n \delta(i_1, \ell)\, T(\ell, i_2,\ldots, i_d),
$$
where $\delta$ is the Kronecker delta function, we have the inequality
$$
\srank(T) \le n.
$$
The notion of slice rank is quite recent, due to Tao \cite{tao}, which he applied 
for a proof of the Croot--Lev--Pach \cite{clp} and Ellenberg--Gijswijt \cite{eg17} theorems on progression-free sets. Tao's lemma stating that the slice rank of diagonal tensors is full 
(given diagonal values are nonzero),  
found many 
applications, 
see \cite{st, blasiak, ns, mt, nas, sau} and  
the survey \cite{gro} with many more references therein. 

\subsubsection*{\bf The partition rank} 
For the {\textit{partition rank}}, which is a refinement of the previous notion, 
simple tensors are decomposable w.r.t. a partition of coordinates, i.e. $T \in \mathsf{S}$ iff it has a form
$$
T(i_1,\ldots, i_d)  = T_1(i_{a_1}, \ldots, i_{a_k}) \cdot T_2(i_{b_1},\ldots, i_{b_{d-k}}), \quad \forall\, i_1,\ldots, i_d \in [n]
$$
for some set partition $\{a_1 < \ldots < a_k \} \cup \{b_1 < \ldots < b_{d-k} \} = [d]$ 
and nonzero tensors $T_1 \in \mathsf{T}^{k}(n), T_2 \in \mathsf{T}^{d - k}(n)$. 
(Like for the slice rank,  partitions can be different for each simple tensor.)
The partition rank of $T \in \mathsf{T}^d(n)$ is denoted by $\mathrm{prank}(T)$. 
The notion of partition rank is due to Naslund \cite{nas} which he used to extend Tao's slice rank method. 
In particular, he showed that the partition rank of diagonal tensors is also full.

Furthermore, we define the {\textit{odd partition rank}}, denoted by $\mathrm{oprank}(T)$, which is the rank function with simple tensors as in the partition rank but conditioned that in a partition there must be an odd size set which does not contain the element~$1$.
For example, $T_1(i_1, i_2) \cdot T_2(i_3, i_4, i_5)$ has odd partition rank $1$ (since the partition has an odd size set $\{3,4,5 \}$ without the element $1$), while the tensor $T_1(i_1, i_2, i_3) \cdot T_2(i_4, i_5)$ can have a larger one.
We clearly have the inequalities
$$
\mathrm{prank}(T) \le \mathrm{oprank}(T) \le \mathrm{rank}(T), \qquad \mathrm{oprank}(T) \le n, 
$$
and additionally, the inequality
$$
\mathrm{oprank}(T) \le \srank(T), \quad \text{ if } d \text{ is even}.
$$

Note that for $d = 2$ all four rank versions in these examples coincide with the usual matrix rank. For $d = 3$, the notions of partition rank and slice rank coincide, and the odd partition rank is at least the slice rank. For $d = 4$, the odd partition rank coincides with the slice rank. 
Note also that 
all four rank versions are indeed different, see Example~\ref{exx} emphasizing this. 

\subsection{Hyperdeterminants}
The {\it hyperdeterminant} of $T \in \mathsf{T}^d(n)$ 
is 
defined as follows
\begin{align*}
\mathrm{det}(T) := 
\sum_{\sigma_{2}, \ldots, \sigma_{d} \in S_n} \mathrm{sgn}(\sigma_{2} \cdots \sigma_{d}) \prod_{i = 1}^n T(i, \sigma_{2}(i), \ldots, \sigma_{d}(i)).
\end{align*}

This definition is classical and was introduced by Cayley \cite{cay}. In literature, it is also referred to  Cayley's first\footnote{There is also {Cayley's second (geometric) hyperdeterminant} \cite{cay2} which we do not discuss here.} hyperdeterminant \cite{eg, cp, jz},  
or combinatorial hyperdeterminant \cite{lim1,lim2}. 
Like the matrix determinant, this function is also relative invariant under multilinear $\mathrm{GL}(n)$ transformations (for even $d$, see Prop.~\ref{propinv}), and it has other good properties which we shall discuss in this paper. 

Even though the given definition 
is nontrivial for all $d$, more significant properties (like full invariance) hold 
only for even $d$ (some reasons are visible in Sec.~\ref{sechyp}).  
For odd~$d$, the hyperdeterminant is usually set to be $0$ but we do not do so. For example, if $d = 3$ then $\det(T)$ 
is actually the {\it mixed discriminant} (e.g. \cite{bapat, gur}), 
see remark~\ref{mixd}. 
There are simple reductions allowing to embed tensors in larger dimensions while preserving  hyperdeterminants, 
 that way we can also obtain {\it hyperpermanents} as a special case, see remark~\ref{remhyp}. 

Many properties of hyperdeterminants and earlier historical notes are discussed in 
Sokolov's books \cite{sok1, sok2}.  
More recently, 
many formulas 
 related to matroids, symmetric functions, random matrices, and stochastic processes were shown to be hyperdeterminantal, see \cite{bar, lt, mats, eg, lv, lammers}.

\subsection{Bounding rank functions via hyperdeterminants} In this paper we relate hyperdeterminants 
with the types of ranks discussed above. More generally, we show that tensors with bounded rank function whose simple tensors satisfy some conditions, have zero hyperdeterminants.  
More specifically, 
we prove 
that hyperdeterminants vanish on tensors having non-full odd partition rank.

\begin{theorem}[=Theorem~\ref{sps}]\label{thmintro1}
Let 
$T\in \mathsf{T}^d(n)$ with $\mathrm{oprank}(T) < n$. Then $\det(T) = 0$.  
Equivalently, $\det(T) \ne 0$ implies that $\mathrm{oprank}(T) = n$ is full; in particular, $ \mathrm{rank}(T) \ge n$, and if $d$ is even then $\srank(T) = n$ is also full.  
\end{theorem}

In particular, this result applies to slice rank if $d$ is even: $\srank(T) < n$ implies $\det(T) = 0$. This property 
is 
a special case of a connection of the slice rank with {\it unstable} tensors from geometric invariant theory as shown in \cite{blasiak};  
in fact $\srank(T) < n$ implies that every $\mathrm{SL}(n)^d$ invariant homogeneous  polynomial vanishes on $T$ (here one works over algebraically closed fields, which is not necessary for our statement on hyperdeterminant\footnote{ The definition of hyperdeterminant and our results still work if we consider tensors as functions $T : [n]^d \to R$ over any commutative ring $R$.}); 
see also \cite{widg} for more on this connection. Note that specifically for hyperdeterminants the odd partition rank strengthens this property. 

An appealing feature of 
hyperdeterminants is that they are explicit and can be a good  
algebraic tool allowing computations using various operations,
see Sec.~\ref{sechyp}, which also includes methods for constructing tensors with nonzero hyperdeterminants.
On the other hand,  
the general problem of deciding if hyperdeterminant vanishes is NP-hard\footnote{As most tensor problems \cite{lim2}.}  
\cite{bar};\footnote{Barvinok \cite{bar} constructs a $4$-tensor whose hyperdeterminant computes the number of Hamiltonian paths in a directed graph. Note also that hyperdeterminants of $4$-tensors include permanents as a special case, see remark~\ref{remhyp}.}
 yet for some classes of (relatively sparse) tensors there are polynomial time algorithms \cite{cp}.

The statement in Theorem~\ref{thmintro1} does not hold for the partition rank, e.g. we  construct tensors $T$ over fields of characteristic $0$ with $\mathrm{prank}(T) = 1$ but $\det(T) \ne 0$, see Example~\ref{exx}.  
However, over fields of positive characteristic, which are of especial interest for applications of the slice rank method, we prove the following analogous result on the partition rank.

\begin{theorem}[=Theorem~\ref{ffthm}]
Let $T\in \mathsf{T}^d(n)$ over a field of characteristic $p > 0$ 
with $\mathrm{prank}(T) < n/(p - 1)$. Then $\det(T) = 0$. Equivalently, $\det(T) \ne 0$ implies $\mathrm{prank}(T) \ge n/(p - 1)$.  
\end{theorem}

Let us also mention that 
hyperdeterminants are related to the {\it Alon--Tarsi conjecture} on latin squares \cite{at}, 
which can be formulated via certain nonzero hyperdeterminants \cite{zappa}.

\subsection{Order polytopes and colored sum-ordered sets} 
We show that `simplest' hyperdeterminants arise for tensors in what we call {\it $P$-echelon form} depending on a poset $P$, see Sec.~\ref{sec:pech}.
Such tensors are supported on lattice points of Stanley's {order polytopes} \cite{sta}. 
For a poset $P = ([d], <_P)$ on $[d]$ the {\it order polytope} $\mathcal{O}_t(P) \subset \mathbb{R}^d$ is defined as follows
$$
\mathcal{O}_t(P) := \left\{(z_1,\ldots, z_d) \in \mathbb{R}^d: 0 \le z_i \le t,\text{ and } z_i \le z_j \text{ for $i <_P j$} \right\}.
$$ 

We define a {\it $d$-colored sum-ordered set} in $\mathbb{F}^n_p$ (for prime $p$)  
of {\it size} $N$ 
as an ordered collection $(x^{(\ell)}_{i})_{\ell \in [d], i \in [N]}$  of vectors $x^{(\ell)}_{i} \in \mathbb{F}^n_p$ 
such that  
$
\sum_{\ell = 1}^d x^{(\ell)}_{i} = 0 
$
for $i \in [N]$, 
and for $i_1,\ldots, i_d \in [N]$ 
$$
\sum_{\ell = 1}^d x^{(\ell)}_{i_{\ell}} = 0
\implies 
(i_1, \ldots, i_d) \in \mathcal{O}_{N}(P)
$$
for some poset $P$ on $[d]$ with a connected Hasse diagram. 
This condition is a generalization of {\it colored sum-free sets} (see \cite{blasiak}) requiring 
$
\sum_{\ell = 1}^d x^{(\ell)}_{i_{\ell}} = 0
$ iff $i_1 = \ldots = i_d$.
Similar to colored sum-free sets  where the slice rank method applies \cite{blasiak}, we show that the maximal size $N$ of a $d$-colored sum-ordered set is exponentially smaller than $p^n$ assuming $d,p$ are constants.

\begin{theorem}[=Theorem~\ref{thm:caps}]\label{capsintro}
Let $d > 2$ 
and $N$ be size of a $d$-colored sum-ordered set in $\mathbb{F}^n_p$. 
Then $N = O(\gamma^n)$ for a constant $1 \le \gamma < p$.
\end{theorem}

This result is a variation on the theme of progression-free sets and 
Tao's slice rank method, 
combined with rank bounds discussed above. 

\section{Rank functions}
Recall that $\mathsf{T}^d(n) = \{T : [n]^{d} \to \mathbb{F} \}$. 
For $\mathsf{S} \subset \mathsf{T}^d(n)$ 
the {rank function} $\mathrm{rank}_{\mathsf{S}} : \mathsf{T}^d(n) \to \mathbb{N}$ 
is defined as 
$$
\mathrm{rank}_{\mathsf{S}}(T) = \min\left\{r\, :\, 
T = \sum_{\ell = 1}^{r} T_{\ell}, \quad T_{\ell} \in \mathsf{S} 
\right\}
$$
assuming 
it is well defined, and $\mathrm{rank}_{\mathsf{S}}(T) = 0$ iff $T = 0$. In particular, $\mathrm{rank}_{\mathsf{S}}(T) = 1$ for $T \in {\mathsf{S}}$. 
Tensors in $\mathsf{S}$ are called {simple} tensors. 

\begin{definition}[Multilinear product]
Let $T \in \mathsf{T}^d(n)$ and $A_1,\ldots, A_d \in \mathsf{T}^2(n)$ be $n \times n$ matrices. Define the {\it multilinear product}
$$
(A_1,\ldots, A_d) \cdot T = T' \in \mathsf{T}^d(n),
$$
where we have
$$
T'(i_1,\ldots, i_d) = \sum_{j_1,\ldots, j_d \in [n]} A_1(i_1,j_1) \cdots A_d(i_d,j_d)\, T(j_1,\ldots, j_d).
$$
This product simply expresses change of basis for a tensor. For $A_1,\ldots, A_d \in \mathrm{GL}(n)$ the multilinear product defines the natural $\mathrm{GL}(n)^d = \mathrm{GL}(n) \times \cdots \times \mathrm{GL}(n)$ ($d$ times) action on $\mathsf{T}^d(n)$.
\end{definition}

For $\mathcal{I} = (I_1,\ldots, I_d)$ where $I_1,\ldots, I_d \subseteq [n]$, we denote by $T_{\mathcal{I}} = (T(i_1,\ldots, i_d))_{i_1 \in I_1, \ldots, i_d \in I_d}$ the {\it subtensor} of $T$ restricted on $\mathcal{I}$. 

\begin{proposition}[Basic properties of rank functions]\label{rankbasic}
The following properties hold:

(i) Monotonicity: for $\mathcal{I} = (I_1,\ldots, I_d)$ with $I_1,\ldots, I_d \subset [n]$ with equal cardinalities 
$$
\mathrm{rank}_{\mathsf{S}_{\mathcal{I}}}(T_{\mathcal{I}}) \le  \mathrm{rank}_{\mathsf{S}}(T), 
$$
where $\mathsf{S}_{\mathcal{I}} = \{T_{\mathcal{I}} : T \in \mathsf{S} \}$ is the restriction of $\mathsf{S}$ on $\mathcal{I}$.

(ii) Sub-additivity:
$$
\mathrm{rank}_{\mathsf{S}}(T + T') \le \mathrm{rank}_{\mathsf{S}}(T) + \mathrm{rank}_{\mathsf{S}}(T').
$$

(iii) Invariance: 
Suppose $\mathsf{S}$ is $\mathrm{GL}(n)^d$ invariant. Then $\mathrm{rank}_{\mathsf{S}}$ is $\mathrm{GL}(n)^d$ invariant, i.e. 
$$
\mathrm{rank}_{\mathsf{S}}((A_1,\ldots, A_d) \cdot T) = \mathrm{rank}_{\mathsf{S}}(T), \qquad A_1,\ldots, A_d \in \mathrm{GL}(n).
$$
\end{proposition}
\begin{proof}
(i) Any decomposition $T = \sum_{\ell = 1}^r T_{\ell}$ where $T_{\ell} \in \mathsf{S}$ can be restricted to $\mathcal{I}$: $T_{\mathcal{I}} = \sum_{\ell = 1}^r (T_{\ell})_{\mathcal{I}}$ where $(T_{\ell})_{\mathcal{I}} \in \mathsf{S}_{\mathcal{I}}$. 
(ii) Take decompositions $T = \sum_{\ell = 1}^r T_{\ell}$ where $T_{\ell} \in \mathsf{S}$ and 
$T' = \sum_{\ell = 1}^{r'} T'_{\ell}$ where $T'_{\ell} \in \mathsf{S}$. 
Hence $\mathrm{rank}_{\mathsf{S}}(T + T') \le r + r'$.
(iii) Let 
$T = \sum_{\ell = 1}^r T_{\ell}$ where $T_{\ell} \in \mathsf{S}$. Using the linearity we have
$$T' := (A_1,\ldots, A_d) \cdot T = (A_1,\ldots, A_d) \cdot \sum_{\ell = 1}^r T_{\ell} = \sum_{\ell = 1}^r (A_1,\ldots, A_d) \cdot T_{\ell}$$
and  $(A_1,\ldots, A_d) \cdot T_{\ell} \in \mathsf{S}$ since $\mathsf{S}$ is $\mathrm{GL}(n)^d$ invariant. 
Hence $\mathrm{rank}_{\mathsf{S}}(T') \le \mathrm{rank}_{\mathsf{S}}(T)$. On the other hand, we also have
$$
T = (A_1^{-1}, \ldots, A_d^{-1}) \cdot T'
$$
which implies $\mathrm{rank}_{\mathsf{S}}(T) \le \mathrm{rank}_{\mathsf{S}}(T')$ and hence we have the equality. 
\end{proof}

\begin{corollary}\label{cor1}
The rank, the slice rank, and the (odd) partition rank are $\mathrm{GL}(n)^d$ invariant.
\end{corollary}
\begin{proof}
We show that the sets of simple tensors for the rank, the slice rank, and the (odd) partition rank are all $\mathrm{GL}(n)^d$ invariant. Hence all these rank functions will be $\mathrm{GL}(n)^d$ invariant. 
To see this, consider a simple tensor $T \in \mathsf{S} \subset \mathsf{T}^d(n)$ with a decomposition 
$$
T(i_1,\ldots, i_d) = T_1(i_{a_1}, \ldots, i_{a_k}) \cdot T_{2}(i_{b_1}, \ldots, i_{b_{d - k}})
$$
for a set partition $\{a_1,\ldots, a_k \} \cup \{b_1,\ldots, b_{d - k} \} = [d]$. Let $A_1,\ldots, A_d \in \mathrm{GL}(n)$ and
$$
T' = (A_1,\ldots, A_d) \cdot T,\qquad  T'_1 = (A_{i_{a_1}},\ldots, A_{i_{a_k}}) \cdot T_1, \qquad T'_2 = (A_{i_{b_1}},\ldots, A_{i_{b_{d-k}}}) \cdot T_2.
$$
Then we have
\begin{align*}
T'(i_1,\ldots, i_d) &= \sum_{j_1,\ldots, j_d} A_{1}(i_1,j_1) \cdots A_d(i_d, j_d)\, T(j_1,\ldots, j_d) \\
&= \sum_{j_{a_1},\ldots, j_{a_k}} A_{1}(i_{a_1},j_{a_1}) \cdots A_d(i_{a_k}, j_{a_k})\, T_1(j_{a_1},\ldots, j_{a_k}) \\
& \qquad\qquad \times \sum_{j_{b_1}, \ldots, j_{b_{d-k}}} A_{1}(i_{b_1},j_{b_1}) \cdots A_d(j_{b_{d-k}}, j_{b_{d-k}})\, T_2(j_{b_1},\ldots, j_{b_{d - k}})\\
&= T'_1(i_{a_1}, \ldots, i_{a_k}) \cdot T'_2(i_{b_1}, \ldots, i_{b_{d - k}}).
\end{align*}
Hence $T'  \in \mathsf{S}$ and $\mathsf{S}$ is $\mathrm{GL}(n)^d$ invariant. 
\end{proof}

\section{Hyperdeterminants}\label{sechyp}
In this section we discuss some basic properties of hyperdeterminants.
Recall that for $d$-tensor $T \in \mathsf{T}^d(n)$ the hyperdeterminant is given by 
$$
\mathrm{det}(T) = 
\sum_{\sigma_{2}, \ldots, \sigma_{d} \in S_n} \mathrm{sgn}(\sigma_{2} \cdots \sigma_{d}) \prod_{i = 1}^n T(i, \sigma_{2}(i), \ldots, \sigma_{d}(i)).
$$
Note that the definition gives a nontrivial function for all $d$, but we will use it predominantly for even $d$ 
due to properties shown below.

\subsection{Characterization via slices} 
We begin with a simple characterization which highlights few important properties.
\begin{definition}[Slices]
Let $T \in \mathsf{T}^d(n)$ be a $d$-tensor. For each {\it direction} $k \in [d]$, we define the {\it parallel} $(d-1)$-dimensional {\it slices} $T^{(k)}_1,\ldots, T^{(k)}_n$ of $T$ given by fixing the $k$-th coordinate: 
$$T^{(k)}_{\ell}(i_1,\ldots, i_{k-1}, i_{k+1}, \ldots, i_d) = T(i_1,\ldots, i_{k-1}, \ell, i_{k+1}, \ldots, i_d), \quad \ell = 1,\ldots, n.$$ 
\end{definition}

Let $I \in \mathsf{T}^d(n)$ be the {\it diagonal identity} tensor given by $I(i_1,\ldots, i_d) = 1$ if $i_1 = \cdots = i_d$, and $I(i_1,\ldots, i_d) = 0$, otherwise. 

\begin{proposition}[Characterization of hyperdeterminants]  
Let $\Delta : \mathsf{T}^d(n) \to \mathbb{F}$ be a function satisfying the following properties: 
\begin{itemize}
\item[(i)] Linearity: $\Delta$ is multilinear in each slice, i.e. if $T, T'$ differ only in a single slice, then 
$$
\Delta(T + T') = \Delta(T) + \Delta(T').
$$
In addition, if $T'$ is obtained from $T$ by multiplying a single slice by a scalar $c \in \mathbb{F}$, then 
$$
\Delta(T') = c\, \Delta(T).
$$
\item[(ii)] Skew-symmetry: If $T'$ is obtained from $T$ by swapping any two 
parallel slices, then $$\Delta(T) = -\Delta(T').$$ 
\item[(iii)] Normalization: For the diagonal identity tensor $I$   
we have 
$$\Delta(I) = 1.$$
\end{itemize}
If $d$ is even, then $\Delta = \det$ is the hyperdeterminant. If $d$ is odd, there is no such function. 
\end{proposition}
\begin{proof}
Let $T \in \mathsf{T}^d(n)$. From multilinearity (i) it follows that if $T$ has a zero slice (i.e. with all zero entries), then $\Delta(T) = 0$. Let us show that there is at most one such function $\Delta$ for each $d$. 
Consider an $\ell$-th slice $T^{(1)}_{\ell}$ of $T$ in the first direction $k = 1$. For each $i_2,\ldots, i_d \in [n]$, denote by $T_{\ell,i_2,\ldots, i_d}$ the $d$-tensor obtained from $T$ by replacing its slice $T^{(1)}_{\ell}$ with all zeros but leaving the single entry $T(\ell, i_2, \ldots, i_d)$ intact. From (i) we have 
\begin{align}\label{dsum}
\Delta(T) = \sum_{i_2,\ldots, i_d \in [n]} \Delta(T_{\ell, i_2,\ldots, i_d})
\end{align}
and each of the tensors $T_{\ell,i_2,\ldots, i_d} \in \mathsf{T}^d(n)$ has the $\ell$-th slice in the direction $k = 1$ filled with all zeros except for a single entry. In the same way, we further proceed to iteratively expand each summand in \eqref{dsum} along any other slice until we get a large sum of the form 
$$
\Delta(T) = \sum_{T'} \Delta(T')
$$
where each summand tensor $T'$ satisfies the following: in {\it each} slice, $T'$ has all zeros except for a single entry. Let us show that $\Delta(T')$ is determined uniquely. The non-zero entries of every such tensor $T'$ are $T'(i, \sigma_2(i), \ldots, \sigma_d(i))$  for $i \in [n]$ and some permutations $\sigma_2, \ldots, \sigma_d \in S_n$. Permuting the slices of $T'$ to get the diagonal form and using the skew-symmetry (ii) we obtain that
$$
\Delta(T') = 
\pm \Delta(D),
$$ 
where $D$ is a diagonal tensor such that $D(i,\ldots, i) = T'(i, \sigma_2(i), \ldots, \sigma_d(i))$ for $i \in [n]$ (and $D(i_1,\ldots, i_d) =0$ unless $i_1 = \ldots = i_d$). Now, using the scaling property in (i) and the normalization $\Delta(I) = 1$ we get 
$$
\Delta(D) = \prod_{i = 1}^n T'(i, \sigma_2(i), \ldots, \sigma_d(i)).
$$
Hence, we can see that $\Delta$ is determined uniquely. 

For even $d$, it is easy to check that the function $\Delta = \det$ satisfies all the conditions (i)-(iii). 

For odd $d$, let us show that there is no function $\Delta$ satisfying the conditions (ii) and (iii) simultaneously. Let us sequentially swap the first two slices of $I$ in each direction $k = 1,\ldots, d$. One can check that after applying this procedure the resulting tensor will be $I$ back again. On the other hand, the sign was changed $d$ times, which means that $\Delta(I) = (-1)^d \Delta(I) = -\Delta(I)$, a contradiction.  
\end{proof}

\begin{remark}
The statement we just proved is very similar to a well-known characterization of determinants, but appears to be new. 
\end{remark}

\begin{remark}
A similar characterization holds for odd $d$ as well, if we change the skew-symmetry condition (ii) by saying that it holds for slices except for the direction $1$, where 
we have the symmetry instead, i.e. swapping parallel slices in this direction does not change $\Delta$. 
\end{remark}

\begin{corollary}
Let $d$ be even. If one slice of $T$ is a scalar multiple of a parallel slice of $T$, then $\det(T) = 0$. In particular, if $T$ has slice rank $1$ and $n > 1$, then $\det(T) = 0$. More generally, if in some direction the parallel slices of $T$ are linearly dependent, then $\det(T) = 0$.
\end{corollary}

\subsection{Invariance} The following well-known property is classical, due to Cayley \cite{cay}. 
\begin{proposition}[Invariance]\label{propinv} 
Let $d$ be even. For $n \times n$ matrices $A_1,\ldots, A_d \in \mathsf{T}^2(n)$ we have 
\begin{align}\label{dprod}
\det((A_1,\ldots, A_d) \cdot T) = \det(A_1) \cdots \det(A_d) \cdot \det(T).
\end{align}
In particular, $\det$ is a relative $\mathrm{GL}(n)^d$ invariant, 
and hence is $\mathrm{SL}(n)^d$ invariant.
\end{proposition}


\begin{corollary}[Diagonal equivalent tensors]
Let $d$ be even and $A_1,\ldots, A_d \in \mathsf{T}^2(n)$. We have 
$$
\det(T) = \det(A_1) \cdots \det(A_d), \qquad T(i_1,\ldots, i_d) = \sum_{\ell = 1}^n A_1(i_1,\ell) \cdots A_d(i_d, \ell). 
$$
\end{corollary}
\begin{proof}
Notice that $T = (A_1,\ldots, A_d) \cdot I$ and hence the result follows from \eqref{dprod}. 
\end{proof}

\begin{remark}
For even $d$, the hyperdeterminant 
is in fact the smallest degree $\mathrm{SL}(n)^d$ invariant homogeneous polynomial; here the space of degree $n$ invariants is one-dimensional and is generated by $\det$ \cite{luq}.
\end{remark}

\subsection{Summation formulas}

Recall that for a tensor $T$ and $I_1,\ldots, I_{d} \subset [n]$, we denote by $T_{I_1,\ldots, I_d}$ the subtensor $(T(i_1,\ldots, i_d))_{i_1 \in I_1,\ldots, i_d \in I_d}$. We also denote $\binom{[n]}{k} := \{S : S \subseteq [n], |S| = k \}$. 

\begin{proposition}[Minor summation formula \cite{lt}]\label{minorsum}
Let $d$ be even. Let $X, Y \in \mathsf{T}^d(n)$ be $d$-tensors. We have 
$$
\mathrm{det}(X + Y) = \sum_{k = 0}^n \sum_{I_1,\ldots, I_{d} \in \binom{[n]}{k}} \epsilon_{I_1,\ldots, I_{d}}\, \mathrm{det}(X_{I_1,\ldots, I_{d}}) \, \mathrm{det}(Y_{\overline{I_1},\ldots, \overline{I_{d}}}),
$$
where $\epsilon_{I_1,\ldots, I_{d}} = \sgn(\pi_{I_1} \cdots \pi_{I_d})$ and $\pi_{I_{\ell}} = (i_1,\ldots, i_k, j_{1}, \ldots, j_{n - k}) \in S_n$ for $I_{\ell} = \{i_1 < \cdots < i_{k} \}$ and $\overline{I_{\ell}} = [n] \setminus I_{\ell}= \{j_1 < \cdots < j_{n - k} \}$.
\end{proposition}

\begin{proof}
	Let 
	$\sigma_1$ be the identity permutation. 	
We have
	\begin{align*}
		\det(X + Y) 
		&= \sum_{\sigma_2,\ldots, \sigma_d \in S^{}_{n}} \sgn(\sigma_2 \cdots \sigma_d) \prod_{i=1}^{n} \left(X({\sigma_{1}(i),\ldots, \sigma_{d}(i)}) + Y({\sigma_{1}(i),\ldots,\sigma_{d}(i)}) \right).
	\end{align*}
	Let us expand all the parentheses and group the terms, 
	such that for $J \in \binom{[n]}{k}$ the term  $$\prod_{j \in J} X({\sigma_{1}(j),\ldots,\sigma_{d}(j))}\prod_{j \not\in J} Y({\sigma_{1}(j),\ldots,\sigma_{d}(j)})$$ corresponds to the group $(I_1,\ldots,I_d)$, where $I_{\ell} = \{\sigma_{\ell}(j): j \in J\}$ for $\ell \in [d]$ (note that $I_{1} = J$).
	Decompose $\sigma_1,\ldots, \sigma_d$ into two tuples of permutations $\sigma^{X} = (\sigma^X_1,\ldots, \sigma^X_d) := (\sigma_{1}|_{J},\ldots,\sigma_{d}|_{J})$ and $\sigma^{Y} = (\sigma^Y_1,\ldots, \sigma^Y_d) :=(\sigma_{1}|_{\overline{J}},\ldots,\sigma_{d}|_{\overline{J}})$. 
	Let us check  that $$\sgn(\sigma_1 \cdots \sigma_d) = \epsilon_{I_{1},\ldots,I_{d}}\, \sgn(\sigma^{X})\, \sgn(\sigma^{Y}),$$ where $\sgn(\sigma^X) = \sgn(\sigma^X_1) \cdots \sgn(\sigma^X_d)$ and $\sgn(\sigma^Y)$ is defined similarly.  
	For  $\ell\in[2,d]$ 
	let us transform the permutation $\sigma_{\ell}$ to $\pi_{I_{\ell}}$ by adjacent transpositions as follows: first transform (i) $\sigma_{\ell} \to \sigma^{X}_{\ell} \sigma^{Y}_{\ell}$ (concatenation) and then (ii) $\sigma^{X}_{\ell}\sigma^{Y}_{\ell} \to \pi_{I_{\ell}}$. We have $\sgn(\sigma_{\ell}) = (-1)^{t_{}}\sgn(\pi_{I_{\ell}}),$ where $t_{}$ is the number of transpositions used during this procedure. 
	Let $t_1$ be the number of transpositions used in (i). 
	Note that $t_1$ depends only on the initial set $J = I_{1}$ and we have $(-1)^{t_1} = \sgn(\pi^{-1}_{I_1}) = \sgn(\pi_{I_1})$. Then (ii) contributes $\sgn(\sigma^{X}_\ell)\, \sgn(\sigma^{Y}_\ell)$. Overall, we have 
	$$\sgn(\sigma_{2}\ldots\sigma_{d}) = (-1)^{t_1(d-1)} \prod_{\ell=2}^{d} \sgn(\pi_{I_{\ell}})\,\sgn(\sigma^{X}_{\ell})\, \sgn(\sigma^{Y}_{\ell}) = \sgn(\pi_{I_{1}}\ldots\pi_{I_{d}})\, \sgn(\sigma^{X})\, \sgn(\sigma^{Y}).$$
	For $\mathcal{I} = (I_{1},\ldots,I_{d}) \in \binom{[n]}{k}^{d}$, denote $S_{}(\mathcal{I}) = \{ (\sigma_{1}(I_{1}),\ldots,\sigma_{d}(I_{d})) : \sigma_2, \ldots, \sigma_d \in S_n\}$. 
	So we have 
	{\small
	\begin{align*}
		\det(X &+ Y) = \sum_{k=0}^{n} 
		\sum_{\mathcal{I} \in \binom{[n]}{k}^d } 
		\sum_{{\sigma^{X} \in S_{}(\mathcal{I}),\atop \sigma^{Y} \in S_{}(\overline{\mathcal{I}})} } 
			\epsilon_{\mathcal{I}}\, \sgn(\sigma^{X})\, \sgn(\sigma^{Y}) 
			\prod_{i=1}^{k} X({\sigma^{X}_{1}(i),\ldots,\sigma^{X}_{d}(i)})
			\prod_{i=1}^{n-k} Y({\sigma^{Y}_{1}(i),\ldots,\sigma^{Y}_{d}(i)}) 
			 \\ 
		&=\sum_{k=0}^{n} 
		\sum_{\mathcal{I} \in \binom{[n]}{k}^d } \epsilon_{\mathcal{I}}  
			\sum_{\sigma^{X} \in S_{}(\mathcal{I})} \sgn(\sigma^{X})
			\prod_{i=1}^{k} X({\sigma^{X}_{1}(i),\ldots,\sigma^{X}_{d}(i)})
		\sum_{\sigma^{Y} \in S_{}(\overline{\mathcal{I}})} \sgn(\sigma^{Y}) 
			\prod_{i=1}^{n-k} Y({\sigma^{Y}_{1}(i),\ldots,\sigma^{Y}_{d}(i)})  
		\\
		&=\sum_{k=0}^{n} 
		\sum_{\mathcal{I} \in \binom{[n]}{k}^d } 
		\epsilon_{\mathcal{I}} \det(X_{\mathcal{I}}) \det(Y_{\overline{\mathcal{I}} } )
	\end{align*}}
	as claimed.
\end{proof}
\begin{corollary}[Laplace expansion \cite{bar, lt}]
Let $d$ be even and $T \in \mathsf{T}^d(n)$. 
We have
\begin{align}\label{lapl}
\det(T) = \sum_{i_2,\ldots, i_d \in [n]} (-1)^{1 + i_2 +  \ldots + i_d}\, T(1, i_2, \ldots, i_d)\, \det(T_{\overline{1}, \overline{i_2}, \ldots, \overline{i_d}}).
\end{align}
\end{corollary}

\subsection{Product formulas} Define the 
{\it outer tensor product} $\otimes : \mathsf{T}^{k}(n) \times \mathsf{T}^{d-k}(n)  \to \mathsf{T}^{d}(n)$ as
$$
T = X \otimes Y, \qquad 
T(i_1,\ldots, i_{d}) = X(i_1,\ldots, i_{k}) \cdot Y(i_{k+1},\ldots, i_{d}).
$$

\begin{proposition}\label{thm:outerprod}
Let $d - k$ be even and 
$X \in \mathsf{T}^k(n), Y \in \mathsf{T}^{d-k}(n)$. Then 
\begin{align}
\det(X \otimes Y) =  \det(X) \cdot \det(Y) \cdot n!
\end{align}
\end{proposition}
\begin{proof}
We have 
{\small
\begin{align*}
\det(X \otimes Y) &= \sum_{{\sigma_2, \ldots, \sigma_k \in S_n \atop \pi_1, \ldots, \pi_{d-k} \in S_n} } \sgn(\sigma_2 \cdots \sigma_k\, \pi_1 \cdots \pi_{d-k})  \prod_{i = 1}^n X(i, \sigma_2(i), \ldots, \sigma_k(i)) \cdot Y(\pi_1(i), \ldots, \pi_{d-k}(i)) \\
&= \left(\sum_{{\sigma_2, \ldots, \sigma_k \in S_n} } \sgn(\sigma_2 \cdots \sigma_k)  \prod_{i = 1}^n X(i, \sigma_2(i), \ldots, \sigma_k(i)) \right) \\
&\qquad\qquad\qquad \times \left(\sum_{{\pi_1, \ldots, \pi_{d-k} \in S_n} } \sgn(\pi_1 \cdots \pi_{d-k}) \prod_{i = 1}^n Y(\pi_1(i),  \ldots, \pi_{d-k}(i)) \right) 
\end{align*}
}
It is easy to see that for even $d-k$ we have 
$$
\sum_{{\pi_1, \ldots, \pi_{d-k} \in S_n} } \sgn(\pi_1 \cdots \pi_{d-k}) \prod_{i = 1}^n Y(\pi_1(i),  \ldots, \pi_{d-k}(i)) = \det(Y) \cdot n!
$$
Hence the formula follows. 
\end{proof}

\subsection{Remarks} 

\subsubsection{More product formulas}\label{moreprod}
There are at least two more product formulas for hyperdeterminants. First, define the {\it contraction product} $\circ : \mathsf{T}^{d}(n) \times \mathsf{T}^{k}(n)  \to \mathsf{T}^{d + k - 2}(n)$ given by
$$
T = X \circ Y, \qquad 
T(i_1,\ldots, i_{d-1}, j_1,\ldots, j_{k-1}) = \sum_{\ell = 1}^n X(i_1,\ldots, i_{d-1}, \ell)\, Y(\ell, j_1,\ldots, j_{k-1}).
$$
Then it is well known (e.g. \cite{sok2, luq}) that for even $d, k$ we have 
$$
\det(X \circ Y) = \det(X) \cdot \det(Y), \qquad X \in \mathsf{T}^d(n), Y \in \mathsf{T}^{k}(n).
$$
Next, define the {\it direct sum} $\oplus : \mathsf{T}^d(n) \times \mathsf{T}^d(m) \to \mathsf{T}^{d}(n + m)$ given by 
$$
T = X \oplus Y, \qquad 
T(i_1,\ldots, i_d) = 
\begin{cases} 
	X(i_1,\ldots, i_d), & \text{ if } i_1,\ldots, i_d \in [n],\\
	Y(i_1,\ldots, i_d), & \text{ if } i_1,\ldots, i_d \in [n+1,\ldots, n + m], \\
	0, & \text{ otherwise}.
\end{cases}
$$
Then for even $d, k$ we also have
$$ 
\det(X \oplus Y) = \det(X) \cdot \det(Y), \qquad X \in \mathsf{T}^d(n), Y \in \mathsf{T}^{d}(m).
$$ 
This formula follows from the minor summation formula in Prop.~\ref{minorsum}.
All the discussed product formulas can be useful for constructing tensors with nonzero hyperdeterminants.

\subsubsection{Reductions and the odd $d$ case} Let $d > 2$ 
and $T \in \mathsf{T}^d(n)$. 
Assume $T(i_1,\ldots, i_d) = 0$ unless  $i_{1} = i_2$. Then it is easy to see that 
$$\det(T) = \det(T'), \qquad T' \in \mathsf{T}^{d - 1}(n), \quad T'(i_2,\ldots, i_{d}) = T(i_2, i_2, i_3, \ldots, i_{d}).$$
In particular, this reduction shows that the properties of hyperdeterminants for even $d$, can be translated to analogous properties for odd $d$. This feature will be useful in the next section.

\subsubsection{Mixed discriminants}\label{mixd} For $n \times n$ matrices $A_1, \ldots, A_n$  the {\it mixed discriminant}  
$D(A_1, \ldots, A_n)$ is  defined as follows 
$$
D(A_1, \ldots, A_n) = \frac{\partial^n}{\partial z_1 \cdots \partial z_n} \det(z_1 A_1 + \ldots + z_n A_n). 
$$
Let 
$T \in \mathsf{T}^{3}(n)$. Then we have
$$
\det(T) = D(T_1, \ldots, T_n), \text{ where } T_{\ell}(i,j) = T(\ell, i, j), \quad \ell, i,j \in [n].
$$

\subsubsection{Hyperpermanents}\label{remhyp} For $X \in \mathsf{T}^k(n)$ the {\it hyperpermanent} $\mathrm{per}(X)$ is defined as follows
$$
\mathrm{per}(X) = \sum_{\sigma_2, \ldots, \sigma_k \in S_n} \prod_{i = 1}^n X(i, \sigma_2(i), \ldots, \sigma_k(i)). 
$$
Let $d$ be even, $T \in \mathsf{T}^d(n)$ and assume $T(i_1,\ldots, i_d) = 0$ unless $i_1 = i_2, i_3 = i_4, \ldots, i_{d-1} = i_d$. Then is easy to see that 
$$\det(T) = \mathrm{per}(Y), \qquad Y \in \mathsf{T}^{d/2}(n), \quad Y(j_1, \ldots, j_{d/2}) = T(j_1, j_1,  \ldots, j_{d/2}, j_{d/2}).$$
In particular, for $d = 4$ we have $\det(T) = \mathrm{per}(Y)$ is the matrix permanent. 
Hyperdeterminants of $4$-tensors 
cover mixed discriminants, permanents, and determinants as special cases.

\section{Bounding rank functions via hyperdeterminants}\label{sec:generic}
\subsection{Generic case} 
We say that $T \in \mathsf{T}^d(n)$ is {\it $k$-null} if $\det(T_{I_1, \ldots, I_d}) = 0$ for all $I_1,\ldots, I_d \in \binom{[n]}{k}$. 

\begin{lemma}\label{generic}
Let $d$ be even. Suppose each simple tensor in $\mathsf{S}$ is $k$-null for $k > 1$.
Let $T \in \mathsf{T}^d(n)$ be a tensor with $\mathrm{rank}_{\mathsf{S}}(T)  < n/(k - 1)$. Then $\det(T) = 0$. 
Equivalently, $\det(T) \ne 0$ implies $\mathrm{rank}_{\mathsf{S}}(T) \ge n/(k-1)$. 
\end{lemma}

Let us first prove the following.
\begin{lemma}\label{dall}
Let $d$ be even and $T \in \mathsf{T}^d(n)$ be $k$-null for $k < n$. Then $T$ is also $(k+1)$-null.
\end{lemma}
\begin{proof}
Let $X = T_{I_1,\ldots, I_d} \in \mathsf{T}^d(k+1)$ for $I_1,\ldots, I_{d} \in \binom{[n]}{k+1}$. Using the Laplace expansion~\eqref{lapl} we have 
$$
\det(X) = \sum_{i_2,\ldots, i_d \in [k+1]} \pm X(1,i_2,\ldots, i_d) \det(X_{\overline{1}, \overline{i_2}, \ldots, \overline{i_d}}).
$$
Since $X$ is $k$-null and $X_{\overline{1}, \overline{i_2}, \ldots, \overline{i_d}} \in \mathsf{T}^{d}(k)$ we have $\det(X_{\overline{1}, \overline{i_2}, \ldots, \overline{i_d}}) = 0$ and hence $\det(X) = 0$. 
\end{proof}

\begin{proof}[Proof of Lemma~\ref{generic}]
We may assume $n \ge k$. By Lemma~\ref{dall} each $S \in \mathsf{S}$ is $m$-null for all $m \in [k,n]$; in particular, $\det(S) = 0$. 
Let $r = \mathrm{rank}_{\mathsf{S}}(T)$ and take a rank attaining decomposition $T = \sum_{\ell = 1}^r S^{(\ell)}$ with $S^{(\ell)} \in \mathsf{S}$. Extending the minor summation formula in Prop.~\ref{minorsum} for many summands it follows that 
\begin{align}\label{ssum1}
\det(T) = \det(S^{(1)} + \ldots + S^{(r)}) = \sum_{\mathcal{I}_1,\ldots, \mathcal{I}_r} \pm \det(S^{(1)}_{\mathcal{I}_1}) \cdots \det(S^{(r)}_{\mathcal{I}_r}),
\end{align} 
where the sum is over partitions $\mathcal{I}_1 \cup \cdots \cup \mathcal{I}_r = [n]^d$ such that $S^{(\ell)}_{\mathcal{I}_{\ell}} \in \mathsf{T}^d(m_{\ell})$ and $m_1 + \ldots + m_r = n$ with $m_{\ell} \ge 0$ for $\ell \in [r]$. Since $r  < n/(k-1)$ for each such partition there is $\ell \in [r]$ with $m_{\ell} \ge k$ (otherwise, $n = m_1 + \ldots + m_r \le (k-1) r$), which gives $\det(S^{(\ell)}_{\mathcal{I}_{\ell}}) = 0$ since $S^{(\ell)}$ is $k$-null. Therefore, every summand in r.h.s. of \eqref{ssum1} is zero and we get $\det(T) = 0$.
\end{proof}

\subsection{Specializing to the odd partition (and slice) rank}
\begin{theorem}\label{sps}
Let 
$T\in \mathsf{T}^d(n)$ with $\mathrm{oprank}(T) < n$. Then $\det(T) = 0$.  
Equivalently, $\det(T) \ne 0$ implies $\mathrm{oprank}(T) = n$; in particular, $ \mathrm{rank}(T) \ge n$, and if $d$ is even, $\srank(T) = n$. 
\end{theorem}

Let us first prove the following lemma.

\begin{lemma}\label{d20}
Let $S \in \mathsf{T}^d(n)$ with $\mathrm{oprank}(S) = 1$. Then $S$ is $2$-null.
\end{lemma}
\begin{proof}
Let $X \in \mathsf{T}^d(2)$ be a nonzero subtensor of $S$. For $n = 2$, the hyperdeterminant can be written in the following simpler form
$$
\det(X) = \sum_{i_2,\ldots, i_d \in [2]} (-1)^{1 + i_2 + \ldots + i_d} X(1, i_2, \ldots, i_d)\, X(2, \overline{i_2}, \ldots, \overline{i_d}), \text{ where } \overline{i} = 3 - i.
$$
Since $\mathrm{oprank}(X) = 1$ we know that $X$ admits a decomposition 
$$
X(i_1,\ldots, i_d) = Y(i_{a_1},\ldots, i_{a_k}) \cdot Z(i_{b_1}, \ldots, i_{b_{d-k}})
$$
for a set partition $\{a_{1} < \ldots < a_{k} \} \cup \{b_{1} < \ldots < b_{d - k} \} = [d]$, tensors $Y \in \mathsf{T}^{k}(2)$, $Z \in \mathsf{T}^{d - k}(2)$, odd $d - k$ and  
$1 = a_1$. 
Consider the map $\phi : [2]^d \to [2]^d$ 
given by 
$$
\phi : (i_1,\ldots, i_d) \mapsto (j_1,\ldots, j_d), 
$$
where
$$
j_{\ell} = 
\begin{cases}
	i_{\ell}, & \text{ if } \ell \in \{a_1,\ldots, a_k \},\\
	\overline{i_{\ell}}, & \text{ if } \ell \in \{ b_1,\ldots, b_{d -k} \}
\end{cases}
$$
(note that $i_1 = j_1$ under this map). Since $d - k$ is odd, the signs $\epsilon_1 = (-1)^{1 + i_2 + \ldots + i_d}$ and $\epsilon_2 = (-1)^{1 + j_2 + \ldots + j_d}$ are opposite. 
Then $\phi$ is a sign-reversing involution acting on monomials of $\det(X)$: 
\begin{align*}
\epsilon_1\, X(1, i_2, \ldots, i_d)\, X(2, \overline{i_2}, \ldots, \overline{i_d})  \longmapsto^{\phi} \epsilon_2\,  X(1, j_2, \ldots, j_d)\, X(2, \overline{j_2}, \ldots, \overline{j_d}),
\end{align*}
more precisely as follows
\begin{align*}
&\epsilon_1\, Y(1, i_{a_2}, \ldots, i_{a_k}) \cdot Z(i_{b_1}, \ldots, i_{b_{d-k}}) \,  Y(2, \overline{i_{a_2}}, \ldots, \overline{i_{a_k}}) \cdot Z(\overline{i_{b_1}}, \ldots, \overline{i_{b_{d-k}}}) \\
 &\longmapsto^{\phi} -\epsilon_1\,  Y(1, i_{a_2}, \ldots, i_{a_k}) \cdot Z(\overline{i_{b_1}}, \ldots, \overline{i_{b_{d-k}}}) \,  Y(2, \overline{i_{a_2}}, \ldots, \overline{i_{a_k}}) \cdot Z(i_{b_1}, \ldots, i_{b_{d-k}}).
\end{align*}
Therefore all terms cancel out and we get $\det(X) = 0$. Hence, $S$ is $2$-null.
\end{proof}

\begin{proof}[Proof of Theorem~\ref{sps}]
{\it Case 1.} Assume $d$ is even. We have $\mathsf{S}$ is the set of simple tensors with the odd partition rank $1$. By Lemma~\ref{d20} each simple tensor in $\mathsf{S}$ is $2$-null.  
Therefore, applying Lemma~\ref{generic} for $k = 2$ we obtain that $\mathrm{rank}_{\mathsf{S}}(T) = \mathrm{oprank}(T) < n$ implies $\det(T) = 0$. 

{\it Case 2.} Assume $d$ is odd. 
Consider the map \, $\widehat{~} : \mathsf{T}^{d}(n) \to \mathsf{T}^{d+1}(n)$ such that for $X \in \mathsf{T}^d(n)$ the tensor $\widehat X \in \mathsf{T}^{d+1}(n)$ is given by
\begin{align}\label{xhat}
\widehat X(i_1,\ldots, i_d, i_{d+1}) = X(i_1,\ldots, i_{d})\, \delta(i_1, i_{d+1}).
\end{align}
It is easy to see that applying this extension map we have 
$$
\det(X) = \det(\widehat{X}).
$$
For odd $d$, if $S \in \mathsf{T}^d(n)$ has odd partition rank~$1$, then $\widehat{S} \in \mathsf{T}^{d+1}(n)$ also has odd partition rank~$1$ (a partition for $S$ has two blocks of even and odd sizes so that the element $1$ is in the even size block, then extending this partition for $\widehat{S}$ the elements $\{1, d+1\}$ will be in the same block and we get two blocks of odd sizes). Therefore, 
$\mathrm{oprank}(T) < n$ implies $\mathrm{oprank}(\widehat{T}) < n$ and since we already proved the statement for even $d+1$, we get $\det(\widehat{T}) = 0 = \det(T)$ as needed. 
\end{proof}

\subsection{Specializing to the partition rank} 
\begin{theorem}\label{ffthm}
Let 
$T\in \mathsf{T}^d(n)$ over a field of characteristic $p > 0$ 
with $\mathrm{prank}(T) < n/(p - 1)$. Then $\det(T) = 0$. Equivalently, $\det(T) \ne 0$ implies $\mathrm{prank}(T) \ge n/(p -1)$.  
\end{theorem}

Again, let us first establish the null condition on simple tensors. 

\begin{lemma}\label{ffprank}
Let $d$ be even and $S \in \mathsf{T}^d(n)$ with $\mathrm{prank}(S) = 1$. Then $S$ is $p$-null.
\end{lemma}
\begin{proof}
If $S$ admits a decomposition into two blocks of odd sizes, then by Lemma~\ref{d20} we already have $S$ is $2$-null and hence it is also $p$-null as $p \ge 2$. Now assume that $S$ is decomposed into two blocks of even sizes $k, d - k$ as follows
$$
S(i_1,\ldots, i_d) = X(i_{a_1}, \ldots i_{a_k}) \cdot Y(i_{b_1}, \ldots, i_{b_{d - k}}), \qquad X \in \mathsf{T}^k(n), Y \in \mathsf{T}^{d - k}(n)
$$
for a set partition $\{a_1,\ldots, a_k \} \cup \{b_1,\ldots, b_{d - k} \} = [d]$ where $a_1 = 1$. 
Take any subtensor $S_{\mathcal{I}} \in \mathsf{T}^d(p)$ of $S$. Let $S' \in \mathsf{T}^d(p)$ be a tensor with indices permuted according to the given set partition
$$S'(i_1,\ldots, i_d) = S_{\mathcal{I}}(i_{a_1}, \ldots, i_{a_k}, i_{b_1}, \ldots, i_{b_{d - k}}).$$
It is easy to see that $\det(S_{\mathcal{I}}) = \det(S')$. 
Note that $S' = X_{\mathcal{I}} \otimes Y_{\mathcal{I}}$. Then using the product formula in Prop.~\ref{thm:outerprod} we have 
$$
\det(S_{\mathcal{I}}) = \det(S') = \det(X_{\mathcal{I}} \otimes Y_{\mathcal{I}}) = \det(X_{\mathcal{I}}) \cdot \det(Y_{\mathcal{I}}) \cdot p! = 0.
$$ 
Therefore, $S$ is $p$-null.
\end{proof}

\begin{proof}[Proof of Theorem~\ref{ffthm}]
{\it Case 1.} 
Assume $d$ is even. By Lemma~\ref{ffprank} simple tensors are $p$-null, and hence the result follows from Lemma~\ref{generic} for $k = p$. 

{\it Case 2.} Assume $d$ is odd. Applying the map \eqref{xhat}, 
if $S \in \mathsf{T}^d(n)$ has partition rank~$1$, then $\widehat{S} \in \mathsf{T}^{d+1}(n)$ also has partition rank~$1$. 
Therefore, $\mathrm{prank}(\widehat{T}) \le \mathrm{prank}(T)$ and since we already proved the statement for even $d+1$, we have $\det(\widehat{T}) =0 = \det(T)$. 
\end{proof}

\begin{remark}
All our proofs 
still work if we consider tensors and hyperdeterminants over commutative rings. 
\end{remark}

\subsection{Some examples} 
To express dependence on $d$, denote by $I_d(i_1,\ldots, i_d) = \delta(i_1,i_2) \cdots \delta(i_1,i_d)$ the diagonal identity $d$-tensor (we assume $I_1(i) = 1$ if $d = 1$). 
The following examples show differences between the discussed rank versions.
\begin{example}\label{exx}
We construct $X \in \mathsf{T}^6(n)$ such that: 
\begin{itemize}

\vspace{0.25em}

\item[(a)] $\mathrm{prank}(X) = 1$ while $\mathrm{oprank}(X) = n$. Take $X = I_2 \otimes I_4$ 
whose partition rank is $1$. 
By Prop.~\ref{thm:outerprod} we have 
$\det(X) = \det(I_2) \cdot \det(I_4) \cdot n! = n!$
and hence over  a field of characteristic $0$, by Theorem~\ref{sps}, the odd partition rank of $X$ is $n$. 

\vspace{0.25em}

\item[(b)] $\mathrm{oprank}(X) = 1$ while $\srank(X) = n$. Take $X = I_3 \otimes I_3$ 
whose odd partition rank is~$1$. 
It can be shown that tensors of such types have slice rank $n$, see \cite[Cor. ~3.6]{sau} (which uses a result from \cite{st}).
Note that in this case $\det(X) = 0$ by Lemma~\ref{d20}. 

\vspace{0.25em}

\item[(c)] $\srank(X) = 1$ while $\mathrm{rank}(X) = n$. Take $X  = I_1 \otimes I_5$ 
whose slice rank is $1$. Since the rank of $I_5$ is $n$, the rank of $X$ is also $n$. 
Note that in general tensor rank can be quite large, e.g. for most tensors $X \in \mathsf{T}^6(n)$ we have 
$\mathrm{rank}(X) \ge n^{5}/6$ while $\srank(X) \le n$. 
\end{itemize}

\end{example}

Let us show that for odd $d$, we can have a situation opposite to (b) from the previous example, 
i.e. there are tensors having slice rank $1$ but full odd partition rank (which is impossible for even $d$ since the odd partition rank is at most the slice rank in this case). 

\begin{example}
Let $d > 1$ be odd and $X = I_1 \otimes I_{d-1} \in \mathsf{T}^d(n)$. 
It is easy to compute $\det(X) = n!$ and hence over a field of characteristic $0$, we have $\mathrm{oprank}(X) = n$ while $\srank(X) = 1$. 
\end{example}

Let us also show that we can have a tensor with full partition rank but zero hyperdeterminant.
\begin{example}
Let $X \in \mathsf{T}^6(2)$ be given by $X(i_1,\ldots, i_6) = 1$, if $(i_1,\ldots, i_6)$ has exactly one $i_{\ell} = 2$ for $\ell \in [6]$, and $X(i_1,\ldots, i_6) = 0$, otherwise. We  have $\det(X) = 0$. Let us check that $\mathrm{prank}(X) \ne 1$. Assume $X(i_1,\ldots, i_6) = Y(i_{a_1}, \ldots, i_{a_k}) \cdot Z(i_{b_1}, \ldots, i_{b_{6 - k}})$ for a set partition $\{a_1, \ldots, a_k\} \cup \{b_{1}, \ldots, b_{6 - k} \} = [6]$ with $k \in [5]$. Note that $Y(2,1,\ldots, 1), Y(1,\ldots, 1) \ne 0$. Since $Y(2,1,\ldots, 1) \cdot Z(2,1,\ldots, 1) = 0$ we get $Z(2,1,\ldots, 1) = 0$. But then $Y(1,\ldots, 1) \cdot Z(2,1,\ldots, 1) = 0$, which is a contradiction. Therefore, $\mathrm{prank}(X) = 2$.
\end{example}

\section{An echelon form}\label{sec:pech}
When a matrix 
is in echelon form, its determinant equals the product of elements on the main diagonal. In this section we discuss a generalization of this property to hyperdeterminants. 

\begin{definition}[$P$-echelon form of tensors]\label{def:echelon}
Let $P = ([d], <_P)$ be a poset on $[d]$ with a partial order $<_P$. Assume the Hasse diagram of $P$ is connected (viewed as undirected graph). 
Let us say that tensor $T \in \mathsf{T}^d(n)$ is {\it in $P$-echelon form} if $T(i_1,\ldots, i_d) = 0$ unless 
$i_a \le i_b$ for $a <_P b$. 
\end{definition}

\begin{example}
Let $d = 4$ and $(2 <_P 1), (2 <_P 3), (4 <_P 3)$. 
Then $T \in \mathsf{T}^4(n)$ is in $P$-echelon form if $T(i_1, i_2, i_3, i_4) = 0$ unless $i_1 \ge i_2 \le i_3 \ge i_4$. 
\end{example}

Recall that 
the {order polytope} $\mathcal{O}_t(P) \subset \mathbb{R}^d$ of a poset $P$ is given by 
$$
\mathcal{O}_t(P) = \left\{(x_1,\ldots, x_d) \in \mathbb{R}^d : 0 \le x_a \le t \text{ for } a \in [d],\, x_{a} \le x_b \text{ for } a <_P b \right\}.
$$
When $T \in \mathsf{T}^d(n)$ is in $P$-echelon form, the support of $T$ are integer points in $\mathcal{O}_n(P) \cap \mathbb{Z}^d$, which also correspond to {\it $P$-partitions} \cite[Ch.~3.15]{ec1}. 

\begin{proposition}\label{echelon}
Let $T \in \mathsf{T}^d(n)$ be in $P$-echelon form. Then 
$$\det(T) = \prod_{i = 1}^n T(i, \ldots, i).$$
\end{proposition}
\begin{proof}
Let us check that when $T$ is in $P$-echelon form, we have 
$$
\prod_{i = 1}^n T(i, \sigma_2(i), \ldots, \sigma_d(i)) = 0
$$
unless $\sigma_2, \ldots, \sigma_d \in S_n$ are identity permutations. Consider any element $a$ comparable with $1$ in $P$, assume $a <_P 1$. 
Then by the echelon condition we have $T(i_1, \ldots, i_a, \ldots) = 0$ unless $i_1 \ge i_a$. Hence $\prod_{i = 1}^n T(i, \ldots, \sigma_a(i), \ldots ) = 0$ unless $i \ge \sigma_a(i)$ for {\it all} $i \in [n]$, and so we must have $\sigma_a = (1, \ldots, n)$. The case of an opposite relation $1 <_P a$ follows similarly as we would have $i \le \sigma_a(i)$ for all $i \in [n]$. We showed that for all  $a$ comparable with $1$ in $P$, the permutation $\sigma_a$ must be the identity. Since the Hasse diagram of $P$ is connected, by repeating the same argument to elements in any connected order (e.g. depth-first-search on the Hasse diagram) we obtain that the same holds for all permutations $\sigma_{2}, \ldots, \sigma_d$.
\end{proof}

\begin{corollary}
Let $T \in \mathsf{T}^d(n)$ be in $P$-echelon form with nonzero diagonal, i.e. $T(i, \ldots, i) \ne 0$ for all $i \in [n]$. Then  $\mathrm{oprank}(T) = n$. 
\end{corollary}

We say that tensors $T, T' \in \mathsf{T}^d(n)$ are {\it equivalent} 
if $T = (A_1,\ldots, A_d) \cdot T'$ for some $A_1, \ldots, A_d \in \mathrm{GL}(n)$. 
Let us say that $T$ is {\it reducible} to a $P$-echelon form if $T$ is equivalent to $T'$ for some $T'$ in $P$-echelon form. 

\begin{problem}
Characterize tensors reducible to a $P$-echelon form. 
\end{problem}

\begin{remark}
Even though tensors in $P$-echelon form have a trivial hyperdeterminant depending only on diagonal elements, 
they 
have  more complicated structure than diagonal tensors; e.g. 
 there are tensors in $P$-echelon form which are not equivalent to diagonal tensors.  
\end{remark}

\section{Colored sum-ordered sets}  
Let $p$ be a prime. Recall from the introduction that a {$d$-colored sum-ordered set} in $\mathbb{F}^n_p$  
of {size} $N$ 
is an ordered collection $X = (x^{(\ell)}_{i})_{\ell \in [d], i \in [N]}$  of vectors $x^{(\ell)}_{i} \in \mathbb{F}^n_p$ 
such that 
$
\sum_{\ell = 1}^d x^{(\ell)}_{i} = 0 
$
for $i \in [N]$, 
and for $i_1,\ldots, i_d \in [N]$ 
$$
\sum_{\ell = 1}^d x^{(\ell)}_{i_{\ell}} = 0
\implies 
(i_1, \ldots, i_d) \in \mathcal{O}_{N}(P)
$$
for some poset $P$ on $[d]$ with a connected Hasse diagram. 
Assuming $d,p$ are constants we show that the maximal size $N$ of a $d$-colored sum-ordered set is exponentially smaller $p^n$.

\begin{theorem}\label{thm:caps}
Let $d > 2$ and $N$ be size of a $d$-colored sum-ordered set in $\mathbb{F}^n_p$.
Then $N = O(\gamma^n)$ for a constant $1 \le \gamma < p$.
\end{theorem}

To prove this result we use Tao's slice rank method with the rank bounds obtained above. 
Let 
$X $ be a $d$-colored sum-ordered set in $\mathbb{F}^n_p$ of size $N$. 
Consider the polynomial $Q : (\mathbb{F}^n_p)^d \to \mathbb{F}_p$ given by 
$$
Q(z_1, \ldots, z_d) = \prod_{i = 1}^n \left(1 -  f_i(z_1,\ldots, z_d)^{p-1} \right), \qquad f_i = z_1(i) + \ldots + z_d(i), 
$$
where a variable $z(i)$ denotes the $i$-th coordinate of the vector $z$.  
Consider the restriction tensor $T = Q |_{X} : [N]^d \to \mathbb{F}_p$ given by 
$$
T(i_1,\ldots, i_d) = Q(x^{(1)}_{i_1}, \ldots, x^{(d)}_{i_d}) = 
\begin{cases}
	1, & \text{ if } \sum_{\ell = 1}^d x^{(\ell)}_{i_{\ell}} = 0, \\ 
	0, & \text{ otherwise}.
\end{cases}
$$

\begin{lemma}
We have: 
$\srank(T) \ge N/(p-1)$. 
\end{lemma}
\begin{proof}
Since $X$ is a $d$-colored sum-ordered set, we have $T(i_1,\ldots, i_d) = 0$ unless $(i_1,\ldots, i_d) \in \mathcal{O}_N(P)$ for a poset $P$ on $[d]$ with a connected Hasse diagram. Hence, $T$ is in $P$-echelon form,  see def.~\ref{def:echelon}. Note that $T(i, \ldots, i) = 1$ for all $i \in [N]$. Therefore, by Prop.~\ref{echelon} we have $\det(T) = 1$ and so by Theorem~\ref{ffprank} we get $\srank(T) \ge \mathrm{prank}(T) \ge N/(p-1)$.
\end{proof}
\begin{remark}
In fact we also proved that $\mathrm{oprank}(T) = N$, and if $d$ is even $\mathrm{srank}(T) = N$, 
but the presented bound is good enough to establish the needed result. 
\end{remark}
 
\begin{lemma}
We have: 
$$\srank(T) \le d \cdot c(n), \qquad c(n) := |\left\{\alpha \in \{0,\ldots, p-1\}^n : |\alpha| \le (p-1)n/d \right\}|.$$
\end{lemma}
\begin{proof}
First,  viewing $Q$ as a $d$-tensor in $\mathsf{T}^d(p^n)$, we have $\srank(T) \le \srank(Q)$. Consider the monomial expansion of the polynomial 
$$
Q(z_1,\ldots, z_d) = \sum_{\alpha_1, \ldots, \alpha_d \in \{0,\ldots, p-1 \}^n } c_{\alpha_1, \ldots, \alpha_d}\, z_1^{\alpha_1} \ldots z_{d}^{\alpha_d},
$$
where we use the notation $z^{\alpha} = z(1)^{\alpha(1)} \cdots z(n)^{\alpha(n)}$ for $\alpha = (\alpha(1), \ldots, \alpha(n)) \in \{0,\ldots,p-1 \}^n$. 
Since $\deg(Q) = (p-1)n$, for each monomial $z_1^{\alpha_1} \ldots z_{d}^{\alpha_d}$ there is an index $k \in [d]$ such that $|\alpha_k| \le (p-1)n/d$. Hence, we can represent $Q$ w.r.t. such indices $k$ as follows
$$
Q = \sum_k \sum_{\alpha_k} z_k^{\alpha_k} Q_{\alpha_k}(z_1,\ldots, z_{k-1}, z_{k+1}, \ldots, z_d)
$$
where each term in this sum has slice rank $1$. Therefore, $Q$ has slice rank at most the number of such terms which is bounded above by $d \cdot c(n)$ 
and the claim follows.
\end{proof}

\begin{lemma} We have: 
$
c(n) \le \gamma^n$ for a constant $1 \le \gamma < p.$
\end{lemma}
\begin{proof}
Let $X_i$ be i.i.d. discrete random variables uniformly distributed on $\{0,\ldots,  p-1\}$. 
Then using Chernoff's bound for $t \ge 0$ we have  
\begin{align*}
\mathbb{P}\left(X_1 + \ldots + X_n \le {(p-1)n}/{d} \right) 
\le \frac{\mathbb{E}[e^{-t(X_1+\ldots + X_n)}]}{e^{-t{(p-1)n}/{d}}} 
= \frac{\gamma(t)^n}{p^n}, \quad \gamma(t) :=  \frac{e^{t(p-1)/d}(1 - e^{-pt})}{1 - e^{-t}},
\end{align*}
which is obtained from 
$\mathbb{E}[e^{-tX}] = \frac{1 - e^{-pt}}{p(1 - e^{-t})}$.
Therefore, 
$$
c(n)  \le \gamma^n, \qquad \gamma = \min_{t \ge 0} \gamma(t).
$$
It is routine to verify that indeed $\min_{t \ge 0} \gamma(t) \in [1, p)$ for $d \ge 3$. 
\end{proof}

The above lemmas now imply that $N \le (p-1)d \cdot \gamma^n$ which proves Theorem~\ref{thm:caps}.  \qed

\section{On ranks additivity under direct sums}
In 1973, Strassen \cite{strassen} conjectured that the tensor rank is additive under direct sums. In~2019, Shitov \cite{shitov} disproved this conjecture. Very recently, Gowers \cite{gow} proved that 
the slice rank is additive under direct sums. We conjecture the same for the odd partition rank and show the following special case.


\begin{proposition}
Let $d$ be even, $X \in \mathsf{T}^d(n),$ $Y \in \mathsf{T}^d(m)$ with $\mathrm{oprank}(X) = r$, $\mathrm{oprank}(Y) = s$. Assume that $X$ is not $r$-null and $Y$ is not $s$-null 
(see Sec.~\ref{sec:generic}). 
Then $\mathrm{oprank}(X \oplus Y) = r + s$. 
\end{proposition}
\begin{proof}
It is easy to see that $\mathrm{oprank}(X \oplus Y) \le r + s$.
Let $X' \in \mathsf{T}^d(r)$ and $Y' \in \mathsf{T}^d(s)$ be subtensors of $X$ and $Y$ having nonzero hyperdeterminants. By Theorem~\ref{sps}, we have $\mathrm{oprank}(X') = r$ and $\mathrm{oprank}(Y') = s$. As discussed in remark~\ref{moreprod} we have $\det(X' \oplus Y') = \det(X') \cdot \det(Y') \ne 0$ and hence $\mathrm{oprank}(X' \oplus Y') = r + s$ is full. But $X' \oplus Y'$ is a subtensor of $X \oplus Y$ and we have $\mathrm{oprank}(X' \oplus Y') \le \mathrm{oprank}(X \oplus Y)$ which gives the result.
\end{proof}

In \cite{gow} the same question is posed about the partition rank with a guess that it is not additive for direct sums. The odd partition rank can be viewed as an intermediate function between partition and slice ranks, and we guess (supported by the special case above) that it is `closer' to the slice rank in this regard. 
More generally, it seems interesting to explore the additivity of rank functions $\mathrm{rank}_{\mathsf{S}}(\cdot)$. 

\section*{Acknowledgements}
We are grateful to Askar Dzhumadil'daev for many inspiring conversations.

\end{document}